\documentclass[a4paper,12pt]{article}
\usepackage{amsfonts}
\usepackage{amssymb}
\usepackage{latexsym}
\usepackage{amsmath}
\usepackage{amsthm}

\newtheorem{theorem}{Theorem}
\newtheorem{lemma}[theorem]{Lemma}

\newtheorem{proposition}[theorem]{Proposition}
\newtheorem{example}{Example}
\newtheorem{definition}{Definition}
\newtheorem{remark}{Remark}

\newcommand{\NN}{\mathbb{N}}
\newcommand{\RR}{\mathbb{R}}
\newcommand{\bsgamma}{\boldsymbol{\gamma}}
\newcommand{\bsalpha}{\boldsymbol{\alpha}}
\newcommand{\bsx}{\boldsymbol{x}}
\newcommand{\bsone}{\boldsymbol{1}}
\newcommand{\bszero}{\boldsymbol{0}}
\newcommand{\uu}{\mathfrak{u}}
\newcommand{\cP}{\mathcal{P}}

\newcommand{\cF}{\mathcal{F}}

\title{Tractability properties of the weighted star discrepancy of regular grids}
\author{Friedrich Pillichshammer\thanks{F. Pillichshammer is supported by the 
Austrian Science Fund (FWF) Project F5509-N26, which is a part of the Special Research Program "Quasi-Monte Carlo Methods:
Theory and Applications". The support of the  Erwin Schr\"odinger International Institute for Mathematics and Physics (ESI) under the thematic programme ``Tractability of High Dimensional Problems and Discrepancy'' is gratefully acknowledged.}}

\date{}

\begin{document}

\maketitle

\begin{abstract}
In this paper we study tractability properties of the weighted star discrepancy with general coefficients of centered regular grids with different mesh-sizes. We give exact characterizations of the weight sequences $(\gamma_j)_{j \ge 1}$ such that the regular grid with different mesh-sizes achieves weak, uniform weak, quasi polynomial, polynomial or strong polynomial tractability for the $\bsgamma$-weighted star discrepancy. For example, a necessary and sufficient condition such that the regular grid with different mesh-sizes achieves weak tractability for the $\bsgamma$-weighted star discrepancy is $\lim_{j \rightarrow \infty}j \gamma_j=0$.
\end{abstract}

\centerline{\begin{minipage}[hc]{130mm}{
}
\end{minipage}}

\section{Introduction}

In this short paper we study the weighted star discrepancy of the {\it centered regular grid} with different mesh-sizes $m_1,\ldots,m_d$ in dimension $d$, given by  $$\Gamma_{m_1,\ldots,m_d} =\left\{\left( \frac{2 \ell_1 +1}{2 m_1},\ldots,\frac{2 \ell_d +1}{2 m_d}\right) \ : \ \ell_j \in \{0,1,\ldots,m_j-1\} \mbox{ for $j=1,\ldots,d$}\right\},$$ where $m_1,m_2,\ldots,m_d \in \NN$ are the mesh-sizes for the single coordinate directions.

It may make wonder why the discrepancy of regular grids is worth to be studied since it is well known that the (classical, i.e., unweighted) star discrepancy of regular grids is far from being of optimal order (see, e.g. \cite[Remark~2.20]{LP14}). However, in several recent papers it turned out, that the use of regular grids with different mesh-sizes in the context of numerical integration and approximation in weighted spaces of infinite smooth functions leads to optimal results (e.g., exponential convergence rates or various notions of tractability); see \cite{DKPW,DLPW,KPW}. These findings motivate the study of the weighted star discrepancy of regular grids. Clearly, in the unweighted case the star discrepancy of regular grids (even with different mesh-sizes) suffers from the curse of dimensionality. But if the weights decay fast enough the question arises if it is possible to peter out the mesh-sizes of the single coordinates fast enough in order to achieve some notion of tractabilty? In the extremal (trivial) case of a weight sequence $(\gamma_j)_{j \ge 1}$ which becomes eventually zero this is certainly possible and it is clear that in this case the weighted star discrepancy of regular grids does not depend on the dimension $d$ - no wonder! Such a behavior is called ``strong polynomial tractability''. 

It is the aim of this short paper to classify the decay of weight sequences $(\gamma_j)_{j \ge 1}$ such that certain notions of tractability for the weighted star discrepancy of regular grids with different mesh-sizes are guaranteed. We point out that we do not aim at finding point sets with ``small'' weighted star discrepancy that can achieve tractability under very mild conditions on the weights (some references in this direction will be given at the end of this paper). We are aware that regular grids are by far not the best choices for such a purpose. Rather, we study how good regular grids with different mesh-sizes perform in the context of tractability of weighted star discrepancy. 

The results will be presented in Section~\ref{sec:res}. In the next section we present some basics about the weighted star discrepancy of general point sets and of centered regular grids and we give the definitions of various notions of tractability. We consider the weighted star discrepancy with general coefficients which is slightly more general than the usual notion with equal coefficients that sum up to one. Furthermore, we remark that the centering of the grids under consideration is of no importance. The same results can be shown for not centered regular grids of the form $$\left\{\left( \frac{\ell_1}{m_1},\ldots,\frac{\ell_d}{m_d}\right) \ : \ \ell_j \in \{0,1,\ldots,m_j-1\} \mbox{ for $j=1,\ldots,d$}\right\}.$$ 

\section{Weighted star discrepancy and tractability}

Let $[d]=\{1,2,\ldots,d\}$ and let $$\bsgamma=\{\gamma_{\uu}\ : \ \emptyset \not=\uu \subseteq [d]\} \subseteq [0,1]$$ be a given set of weights.
The $\bsgamma$-weighted star discrepancy of an $N$-point set $\cP_d$ in $[0,1)^d$, introduced by Sloan and Wo\'{z}niakowski \cite{SW98}, is intimately linked to the worst-case integration error of quasi-Monte Carlo (QMC) rules of the form $$\frac{1}{N} \sum_{\bsx \in \cP_d} f(\bsx)$$ for functions $f$ from the weighted function class $\mathcal{F}_{d,1,\bsgamma}$, which is given as follows: Let $\mathcal{W}_1^{(1,1,\ldots,1)}([0,1]^d)$ be the Sobolev space of functions defined on $[0,1]^d$ that are once differentiable in each variable, and whose derivatives have finite $L_1$ norm. Then $$\cF_{d,1,\bsgamma}=\{f \in \mathcal{W}_1^{(1,1,\ldots,1)}([0,1]^d) \ : \ \|f\|_{d,1,\bsgamma}< \infty\},$$ where $$\|f\|_{d,1,\bsgamma} =|f(\bsone)| + \sum_{\emptyset \not=\uu \subseteq [d]} \frac{1}{\gamma_\uu}\left\|\frac{\partial^{|\uu|}}{\partial \bsx_{\uu}}f(\bsx_{\uu},\bsone)\right\|_{L_1}.$$ The fundamental error estimate is a weighted version of the Koksma-Hlawka inequality, see \cite[p.~65]{NW10}. In fact, the worst-case error of a QMC rule in $\cF_{d,1,\bsgamma}$ is exactly the $\bsgamma$-weighted star discrepancy of the point set used in the QMC rule. 

Here we consider more general linear algorithms of the form 
\begin{equation}\label{linrul}
\sum_{\bsx \in \cP_d} a_{\bsx} f(\bsx)
\end{equation}
 with arbitrary coefficients $a_{\bsx} \in \RR$ for $\bsx \in \cP_d$. Then the {\it local discrepancy} of an $N$-point set ${\cal P}_d$ in $[0,1)^d$ with coefficients ${\cal A}(\cP_d)=\{a_{\bsx} \ : \ \bsx \in \cP_d\}$ 
is defined as $$\Delta_{\cP_d,{\cal A}(\cP_d)}(\bsalpha):=\sum_{\bsx \in \cP_d} a_{\bsx} \bsone_{[\bszero,\bsalpha)}(\bsx)-{\rm Volume}([\bszero,\bsalpha))$$
for all $\bsalpha=(\alpha_1,\ldots,\alpha_d) \in [0,1]^d$, where $[\bszero,\bsalpha)=[0,\alpha_1)\times [0,\alpha_2)\times \ldots \times [0,\alpha_d)$ and $\bsone_{[\bszero,\bsalpha)}$ is the characteristic function of this interval (see \cite[Eq. (9.2)]{NW10}, where also this general notion of discrepancy is considered). If $a_{\bsx}=1/N$ for all $\bsx \in \cP_d$ we speak about QMC coefficients; in this case we simply write $\Delta_{{\cal P}_d}$ for the local discrepancy.


\begin{definition}[Weighted star discrepancy]\rm
For an $N$-point set $\cP_d$ in $[0,1)^d$ with coefficients ${\cal A}(\cP_d)$
the {\em $\bsgamma$-weighted star discrepancy} is defined as
$$D_{N,{\bsgamma}}^*({\cal P}_d,{\cal A}(\cP_d)):=\sup_{\bsalpha\in [0,1]^d}
\max_{\emptyset\ne {\mathfrak u}\subseteq [d]} \gamma_{\mathfrak
  u}|\Delta_{\cP_d,{\cal A}(\cP_d)}((\bsalpha_{\mathfrak u},\bsone))|,$$
where for $\bsalpha=(\alpha_1,\ldots,\alpha_d) \in [0,1]^d$ and for
$\uu \subseteq [d]$ we put $(\bsalpha_{\uu},\bsone)=(y_1,\ldots,y_d)$
with $y_j=\alpha_j$ if $j \in \uu$ and $y_j=1$ if $j \not\in \uu$. For QMC coefficients $a_{\bsx}=1/N$ we simply write $D_{N,\bsgamma}^*(\cP_d)$.
\end{definition}

\begin{remark}\rm
\begin{itemize}
\item The weights $\gamma_{\uu}$ model the importance of the groups of coordinate projections with indices in $\uu$ for the integration problem. See \cite{SW98} and \cite[Chapter~9]{NW10} for further information. The $\bsgamma$-weighted star discrepancy $D_{N,{\bsgamma}}^*({\cal P}_d,{\cal A}(\cP_d))$ is exactly the worst-case error of the linear algorithm \eqref{linrul} in $\cF_{d,1,\bsgamma}$.
\item A popular choice for the weights are {\it product weights} given by a sequence $(\gamma_j)_{j \ge 1}$ of positive reals. Then for $\emptyset \not=\uu \subseteq [d]$ one defines 
\begin{equation}\label{def:prodweight}
\gamma_{\uu}=\prod_{j\in \uu} \gamma_j.
\end{equation}
\item If $\gamma_{[d]}=1$ and $\gamma_{\uu}=0$ for all $\uu \varsubsetneqq [d]$, or likewise, in the case of product weights, if $\gamma_j=1$ for all $j \in \NN$ we obtain the classical, i.e., unweighted star discrepancy $D_N^*(\cP_d,{\cal A}(\cP_d))$ which we simply call star discrepancy. Again, for QMC coefficients we write $D_N^*(\cP_d)$.
\end{itemize}
\end{remark}

\begin{lemma}
Let $\cP_d$ be an $N$-point set in $[0,1)^d$ with QMC coefficients $a_{\bsx}=1/N$. For $\emptyset \not= \uu \subseteq [d]$ let $\cP_d(\uu)$ be the $|\uu|$-dimensional point set consisting of the projection of the points in $\cP_d$ to the components given in $\uu$. Then 
\begin{equation}\label{fo:wdisc1}
D_{N,{\bsgamma}}^*({\cal P}_d) = \max_{\emptyset\ne {\mathfrak u}\subseteq [d]} \gamma_{\uu} D_N^{\ast}(\cP_d(\uu)).
\end{equation}
\end{lemma}

\begin{proof}
Obviously, for any $\uu \subseteq [d]$ and any $\bsalpha \in [0,1]^d$ we have $|\Delta_{\cP_d}((\bsalpha_{\mathfrak u},\bsone))| \le D_N^{\ast}(\cP_d(\uu))$ and hence $$D_{N,{\bsgamma}}^*({\cal P}_d) \le \max_{\emptyset\ne {\mathfrak u}\subseteq [d]} \gamma_{\uu} D_N^{\ast}(\cP_d(\uu)).$$ On the other hand, for every $\emptyset\ne {\mathfrak u}\subseteq [d]$ we have $$D_{N,{\bsgamma}}^*({\cal P}_d) \ge \sup_{\bsalpha\in [0,1]^d} \gamma_{\uu} |\Delta_{\cP_d}((\bsalpha_{\mathfrak u},\bsone))| = \gamma_{\uu} D_N^{\ast}(\cP_d(\uu))$$ and hence \eqref{fo:wdisc1} is shown.
\end{proof}

The next proposition gives formulas and estimates for the weighted star discrepancy of regular grids $\Gamma_{m_1,\ldots,m_d}$. Let $N=|\Gamma_{m_1,\ldots,m_d}|=m_1\cdots m_d$.

\begin{proposition}
For the weighted star discrepancy of the centered regular grid $\Gamma_{m_1,\ldots,m_d}$ with arbitrary coefficients $\mathcal{A}(\Gamma_{m_1,\ldots,m_d})$ and product weights \eqref{def:prodweight} for every $\ell \in [d]$ we have 
\begin{equation}\label{lbd_belw}
D_{N,\bsgamma}^*(\Gamma_{m_1,\ldots,m_d},{\cal A}(\Gamma_{m_1,\ldots,m_d})) \ge \frac{\gamma_{\ell}}{4 m_\ell}.
\end{equation}
For the star discrepancy of the centered regular grid $\Gamma_{m_1,\ldots,m_d}$ with QMC coefficients we have 
\begin{equation}\label{fo:disc}
D_N^*(\Gamma_{m_1,\ldots,m_d})=1-\prod_{j=1}^d\left(1-\frac{1}{2 m_j}\right).
\end{equation}
For the weighted star discrepancy of the centered regular grid $\Gamma_{m_1,\ldots,m_d}$ with QMC coefficients and weights $\bsgamma$ we have 
\begin{align}\label{fo:wdisclat}
D_{N,\bsgamma}^*(\Gamma_{m_1,\ldots,m_d}) 
&= \max_{\emptyset \not=\uu \subseteq [d]}  \gamma_{\uu} \left(1-\prod_{j \in \uu}\left(1-\frac{1}{2 m_j}\right)\right),
\end{align}
and 
\begin{align}\label{bd:wdisclat}
D_{N,\bsgamma}^*(\Gamma_{m_1,\ldots,m_d}) \le \max_{\emptyset \not=\uu \subseteq [d]}  \gamma_{\uu} \sum_{j \in \uu} \frac{1}{2 m_j}. 
\end{align}
\end{proposition}

\begin{proof}
We have 
\begin{align*}
D_{N,\bsgamma}^*(\Gamma_{m_1,\ldots,m_d},{\cal A}(\Gamma_{m_1,\ldots,m_d})) \ge & \sup_{\alpha_{\ell} \in[0,1]} \gamma_{\ell} |\Delta_{\Gamma_{m_1,\ldots,m_d},{\cal A}(\Gamma_{m_1,\ldots,m_d})}((1,\ldots,1,\alpha_{\ell},1,\ldots,1))|  \\
\ge & \gamma_{\ell} \max_{0 \le c \le 1} \left|A_{m_1,\ldots,m_d} -1+\frac{c}{2 m_{\ell}}\right|,
\end{align*}
where 
$A_{m_1,\ldots,m_d}=\sum_{\bsx\in \Gamma_{m_1,\ldots,m_d}}a_{\bsx}$ and $c \in [0,1]$. The last inequality can be easily seen by setting $\alpha_{\ell}=1-\frac{c}{2 m_{\ell}}$ for $c\in[0,1]$. 

If $$A_{m_1,\ldots,m_d} \ge 1-\frac{1}{4m_\ell} \ \ \mbox{or}\ \ A_{m_1,\ldots,m_d} \le 1-\frac{3}{4m_\ell},$$ then we have $$\left|A_{m_1,\ldots,m_d}-1+\frac{1}{2 m_\ell}\right| \ge \frac{1}{4 m_\ell}.$$ If $$1-\frac{3}{4m_\ell} \le A_{m_1,\ldots,m_d} \le 1-\frac{1}{4m_\ell},$$ then $$\left|A_{m_1,\ldots,m_d}-1\right| \ge \frac{1}{4 m_\ell}.$$ This implies \eqref{lbd_belw}.

For the proof of formula \eqref{fo:disc} just follow the proof of \cite[Theorem~2.19]{LP14}, which is a special case of the proposed formula. Formula \eqref{fo:wdisclat} follows from \eqref{fo:disc} and \eqref{fo:wdisc1}. The bound \eqref{bd:wdisclat} follows from \eqref{fo:wdisclat} and the fact that for $t \in \mathbb{N}$ and for $\varepsilon_j \in (0,1]$ we have 
\begin{equation*}
1-\prod_{j=1}^t (1-\varepsilon_j) \le \sum_{j=1}^t \varepsilon_j.
\end{equation*}
This can be shown by induction on~$t$. 
\end{proof}

\begin{definition}\rm
For $\varepsilon \in (0,1)$ and $d \in \mathbb{N}$ let 
\begin{eqnarray*}
N_{\bsgamma}(\varepsilon,d):=\min\{N\in \mathbb{N} \ & : &  \ N=m_1\cdots m_d \mbox{ and there exist coefficients ${\cal A}(\Gamma_{m_1,\ldots,m_d})$} \\
& &  \mbox{ such that }  D_{N,\bsgamma}^{\ast}(\Gamma_{m_1,\ldots,m_d},{\cal A}(\Gamma_{m_1,\ldots,m_d})) \le \varepsilon\}.
\end{eqnarray*}
\end{definition}
We are interested in the behavior of $N_{\bsgamma}(\varepsilon,d)$ for $\varepsilon \rightarrow 0$ and $d \rightarrow \infty$. This is the subject of tractability. An overview on the current state of the art of tractability theory can be found in the three volumes \cite{NW08,NW10,NW12}.

\begin{definition}[Tractability]\rm
We deal with the following notions:
\begin{itemize}
\item If there exist positive numbers $C,\varepsilon_0$ and $\tau$ such that $$N_{\bsgamma}(\varepsilon,d) \ge C(1+\tau)^d \ \ \mbox{for all $\varepsilon \le \varepsilon_0$ and infinitely many $d \in \NN$,}$$ then the $\bsgamma$-weighted star discrepancy of the centered regular grid suffers from the {\it curse of dimensionality}.

\item We say that the centered regular grid with different mesh-sizes achieves {\it weak tractability (WT)} for the $\bsgamma$-weighted star discrepancy, if 
\begin{equation}\label{defWT}
\lim_{\varepsilon^{-1}+d \rightarrow \infty}  \frac{\log N_{\bsgamma}(\varepsilon,d)}{\varepsilon^{-1}+d}=0.
\end{equation}
\item We say that the centered regular grid with different mesh-sizes achieves {\it uniform weak tractability (UWT)} for the $\bsgamma$-weighted star discrepancy, if 
\begin{equation}\label{defUWT}
\lim_{\varepsilon^{-1}+d \rightarrow \infty}  \frac{\log N_{\bsgamma}(\varepsilon,d)}{\varepsilon^{-t_1}+d^{t_2}}=0 \ \ \ \mbox{for all $t_1,t_2\in (0,1]$}.
\end{equation}
\item We say that the centered regular grid with different mesh-sizes achieves {\it quasi polynomial tractability (QPT)} for the $\bsgamma$-weighted star discrepancy, if there exist positive numbers $C$ and $t$, such that 
\begin{equation}\label{def:QPT}
N_{\bsgamma}(\varepsilon,d) \le C \exp(t(1+\log d)(1+\log \varepsilon^{-1}))\ \ \ \mbox{ for all $d \in \NN$ and for all $\varepsilon \in (0,1)$.}
\end{equation}
\item We say that the centered regular grid with different mesh-sizes achieves {\it polynomial tractability (PT)} for the $\bsgamma$-weighted star discrepancy, if there exist positive numbers $C, \tau$ and a non-negative $\sigma$, such that 
\begin{equation}\label{defPT}
N_{\bsgamma}(\varepsilon,d)\le C \varepsilon^{-\tau} d^{\sigma} \ \ \ \mbox{for all $d \in \NN$ and for all $\varepsilon \in (0,1)$.}
\end{equation}
If \eqref{defPT} holds with $\sigma=0$, then we speak about {\it strong polynomial tractability (SPT)}.
\end{itemize}
\end{definition}

In the next sections we characterize the weight sequences for which these notions of tractability can be achieved. The results for WT and UWT are presented in Theorem~\ref{thm:WTiff}, the results for QPT, PT and SPT in Theorem~\ref{thm:PTiff}.

\section{The results}\label{sec:res}

In the following we will be concerned only with product weights \eqref{def:prodweight}. Throughout we tacitly assume that the weight sequence is non-increasing and $\gamma_1 \le 1$, i.e. $1 \ge \gamma_1 \ge \gamma_2 \ge \gamma_3 \ge \ldots \ge 0.$

It is clear that in order to achieve positive results for tractability the weights $\gamma_j$ necessarily have to tend to zero. This is the essence of the following result: 
\begin{lemma}
Assume that $\gamma_j \ge c>0$ for all $j \in \NN$. Then for all $\varepsilon \in (0,1]$ and all $d \in \NN$ we have
\begin{equation}\label{LB:bdw}
N_{\bsgamma}(\varepsilon,d) \ge \left(\frac{c}{4\varepsilon}\right)^d .
\end{equation}
In particular, the $\bsgamma$-weighted star discrepancy of the centered regular grid suffers from the curse of dimensionality. In particular, this holds true for the classical star discrepancy.
\end{lemma}

\begin{proof}
According to \eqref{lbd_belw}, for every $j \in [d]$ we have $D_{N,\bsgamma}^{\ast}(\Gamma_{m_1,\ldots,m_d},{\cal A}(\Gamma_{m_1,\ldots,m_d})) \ge \tfrac{c}{4 m_j}$. Thus, if $D_{N,\bsgamma}^{\ast}(\Gamma_{m_1,\ldots,m_d}) \le \varepsilon$, then $m_j \ge \tfrac{c}{4\varepsilon}$ and hence $m_1\cdots m_d \ge \left(\frac{c}{4\varepsilon}\right)^d$. This implies \eqref{LB:bdw}. For example, for $\varepsilon \in (0,\tfrac{c}{6}]$ we have $N_{\bsgamma}(\varepsilon,d) \ge (1.5)^d$ for all $d \in \NN$ and this completes the easy proof. 
\end{proof}

The following theorem gives an exact characterization of weight sequences which lead to WT and UWT, respectively.

\begin{theorem}\label{thm:WTiff}
The centered regular grid with different mesh-sizes achieves 
\begin{enumerate}
\item WT for the $\bsgamma$-weighted star discrepancy if and only if 
\begin{equation}\label{WT:iffbed}
\lim_{j \rightarrow \infty} j \gamma_j=0.
\end{equation}
\item UWT for the $\bsgamma$-weighted star discrepancy if and only if 
\begin{equation}\label{UWT:iffbed}
\lim_{j\rightarrow \infty} j^n \gamma_j=0 \ \ \ \mbox{for all $n \in \NN$.}
\end{equation}
\end{enumerate}
\end{theorem}

\begin{proof}
\begin{enumerate}
\item First we show the necessity of the condition: For every $j \in [d]$ we have $$D_{N,\bsgamma}^{\ast}(\Gamma_{m_1,\ldots,m_d},{\cal A}(\Gamma_{m_1,\ldots,m_d})) \ge \frac{\gamma_j}{4m_j}.$$ Thus, if $D_{N,\bsgamma}^{\ast}(\Gamma_{m_1,\ldots,m_d}) \le \varepsilon$, then $m_j \ge \frac{\gamma_j}{4 \varepsilon}$ and hence $m_1\cdots m_d \ge \left(\frac{1}{4 \varepsilon}\right)^d  \prod_{j=1}^d \gamma_j \ge  \left(\frac{\gamma_d}{4 \varepsilon}\right)^d$. Hence 
\begin{equation}\label{logN:lbd}
\log N_{\bsgamma}(\varepsilon,d) \ge d \log \left(\frac{\gamma_d}{4 \varepsilon}\right).
\end{equation}
Since we have WT, we have $\lim_{\varepsilon^{-1}+d \rightarrow \infty} \frac{\log N_{\bsgamma}(\varepsilon,d)}{\varepsilon^{-1}+d}=0.$ In particular, choosing $\varepsilon^{-1}=\eta d$ for arbitrary large $\eta>0$ and letting $d \rightarrow \infty$ we have $$0=\lim_{d \rightarrow \infty} \frac{d \log\left(\gamma_d d \frac{\eta}{4}\right)}{d(\eta+1)}=\lim_{d \rightarrow \infty} \frac{\log\left(\gamma_d d \frac{\eta}{4}\right)}{\eta+1}.$$ Hence $$\lim_{d \rightarrow \infty} d \gamma_d = \frac{4}{\eta}$$ and \eqref{WT:iffbed} follows, since $\eta$ can be arbitrary large.   

Now assume that \eqref{WT:iffbed} holds. Then there exists some $C>0$ such that $\gamma_j \le C j^{-1}$ for all $j \in \NN$. W.l.o.g. we may assume that $C \ge 1$.  Choose $\Gamma_{m_1,\ldots,m_d}$ with $$m_j=\left\lceil \frac{\gamma_j 3 C {\rm e}^C}{4 \varepsilon}\right\rceil \ \ \ \mbox{ for all $j \in [d]$},$$ and QMC coefficients. Then for $\emptyset \not= \uu \subseteq [d]$ we have 
$$\gamma_{\uu} \sum_{j \in \uu} \frac{1}{2 m_j} \le \gamma_{\uu} \sum_{j \in \uu} \frac{1}{2\, \frac{\gamma_j 3 C {\rm e}^C}{4 \varepsilon}}= \frac{2 \varepsilon}{3 C {\rm e}^C} \sum_{j\in \uu} \prod_{k \in \uu \atop k\not=j} \gamma_k.$$ 
Since the weight sequence is non-increasing it follows that 
\begin{align*}
\sum_{j\in \uu} \prod_{k \in \uu \atop k\not=j} \gamma_k \le & \sum_{j=1}^{|\uu|} \prod_{k=1 \atop k \not=j}^{|\uu|} \gamma_k \le \sum_{j=1}^{|\uu|} \prod_{k=1 \atop k \not=j}^{|\uu|} \frac{C}{k} = \left(\prod_{k=1}^{|\uu|} \frac{C}{k}\right) \sum_{j=1}^{|\uu|} \frac{j}{C} = \frac{C^{|\uu|}}{|\uu|!}  \frac{|\uu|(|\uu|+1)}{2 C}.
\end{align*}

If $|\uu|\ge 2$, then we have $$\frac{C^{|\uu|}}{|\uu|!}  \frac{|\uu|(|\uu|+1)}{2 C}\le \frac{3C}{2} \frac{C^{|\uu|-2}}{(|\uu|-2)!} \le \frac{3C {\rm e}^C}{2} .$$ If $|\uu|=1$, then $$\frac{C^{|\uu|}}{|\uu|!}  \frac{|\uu|(|\uu|+1)}{2 C}=1\le \frac{3C {\rm e}^C}{2}.$$ In any case, for $\emptyset \not= \uu \subseteq [d]$ we have 
$$\gamma_{\uu} \sum_{j \in \uu} \frac{1}{2 m_j} \le \varepsilon$$ and hence, according to \eqref{bd:wdisclat}, $$D_{N,\bsgamma}^*(\Gamma_{m_1,\ldots,m_d}) \le \varepsilon.$$ 

Obviously, $m_j=1$ if 
\begin{equation}\label{bedm1}
\frac{\gamma_j 3 C {\rm e}^C}{4 \varepsilon} \le 1 \ \Leftrightarrow \ \gamma_j \le \frac{4\varepsilon}{3 C {\rm e}^C}.
\end{equation}
Let now $\delta >0$ be arbitrary. According to \eqref{WT:iffbed} there exists a positive number $j_*=j_*(\delta)$ such that for all $j\ge j_*$ we have $$\gamma_j \le \frac{\delta}{j}.$$ Hence, for $j \ge j_*$ condition \eqref{bedm1} is certainly satisfied if $$\frac{\delta}{j} \le \frac{4\varepsilon}{3 C {\rm e}^C} \ \Leftrightarrow \ j \ge \frac{\delta 3 C {\rm e}^C}{4 \varepsilon}.$$  
Let $M:=3 C {\rm e}^C/(4\varepsilon)$ and $N=m_1\cdots m_d$. Then we have 
\begin{align*}
N \le & \prod_{j=1}^{j_* -1} \lceil \gamma_j M\rceil \prod_{j=j_*}^{ \lfloor \delta M\rfloor} \lceil \gamma_j M \rceil \le \lceil \gamma_1 M \rceil^{j_*}  \prod_{j=j_*}^{ \lfloor \delta M\rfloor} \left\lceil \frac{\delta M}{j} \right\rceil  \\
\le & \lceil \gamma_1 M \rceil^{j_*}  \prod_{j=1}^{ \lfloor \delta M\rfloor} \frac{2 \delta M}{j} =   \lceil \gamma_1 M \rceil^{j_*}   \frac{(2 \delta M)^{ \lfloor \delta M\rfloor}}{\lfloor \delta M\rfloor !}\\
\le & \lceil \gamma_1 M \rceil^{j_*}   {\rm e}^{2 \delta M}.
\end{align*} 
Hence $$\log N_{\bsgamma}(\varepsilon,d) \le \log N \le  j_* \log \lceil \gamma_1 M \rceil + 2 \delta M =  j_* \log \left\lceil \gamma_1 \frac{3 C {\rm e}^C}{4 \varepsilon} \right\rceil  + \frac{\delta 3 C {\rm e}^C}{2\varepsilon}$$  
and this implies $$\limsup_{\varepsilon^{-1}+d \rightarrow \infty} \frac{\log N_{\bsgamma}(\varepsilon,d)}{\varepsilon^{-1}+d}\le \delta \frac{3 C {\rm e}^C}{2}.$$ Since $\delta>0$ can be arbitrary small we obtain  $$\lim_{\varepsilon^{-1}+d \rightarrow \infty} \frac{\log N_{\bsgamma}(\varepsilon,d)}{\varepsilon^{-1}+d}=0$$ and thus the centered regular grid with different mesh-sizes achieves WT for the weighted star discrepancy.
\item Using \eqref{logN:lbd} and the assumption of UWT we obtain $\lim_{\varepsilon^{-1}+d \rightarrow \infty} \frac{d\log (\gamma_d/(4\varepsilon))}{\varepsilon^{-t_1}+d^{t_2}}=0$ for all $t_1,t_2 \in (0,1]$. In particular, choosing $t_1=1/n$ for $n \in \NN$ and $\varepsilon^{-1}=(\eta d)^n$ for arbitrary large $\eta>0$ and letting $d \rightarrow \infty$ we have $$0=\lim_{d \rightarrow \infty} \frac{d \log\left(\frac{\gamma_d}{4} (d \eta)^n\right)}{d \eta +d^{t_2}}=\lim_{d \rightarrow \infty} \frac{\log\left(\frac{\gamma_d}{4} (d \eta)^n\right)}{\eta+d^{t_2-1}}.$$ Hence $\lim_{d \rightarrow \infty} d^n \gamma_d = \frac{4}{\eta^n}.$ Since $\eta$ can be arbitrary large it follows that $$\lim_{d \rightarrow \infty} d^n \gamma_d = 0 \ \ \ \mbox{for all $n\in \NN$.}$$

Now assume that \eqref{UWT:iffbed} holds. Let $t_1,t_2 \in (0,1]$ and choose $n \in \NN$ such that $2/n < t_1$. Then there exists a $C>0$ such that $\gamma_j \le C/j^n$ for all $j \ge 1$. We take $\Gamma_{m_1,\ldots,m_d}$ with QMC coefficients where for $j \in [d]$ we choose $$m_j= \left \lceil \frac{\sqrt{\gamma_j}}{\delta}\right\rceil, \ \mbox{where $\delta:=\frac{2 \varepsilon}{S}$ and $S=\sum_{j \ge 1}\sqrt{\gamma_j} < \infty$.}$$ Then we have $D_{N,\bsgamma}^*(\Gamma_{m_1,\ldots,m_d}) \le \varepsilon$. Furthermore, $m_j=1$ if $\frac{\sqrt{\gamma_j}}{\delta} \le 1$ and this is certainly satisfied if $$\frac{\sqrt{C}}{j^{n/2} \delta} \le 1 \ \Leftrightarrow \ j \ge \left(\frac{C}{\delta^2}\right)^{1/n}. $$ With $j_*:=\lfloor \left(\frac{C}{\delta^2}\right)^{1/n} \rfloor$ and $N=m_1 m_2\cdots m_d$ we obtain
\begin{align*}
N & \le  \prod_{j=1}^{j_*}  \left \lceil \frac{\sqrt{\gamma_j}}{\delta}\right\rceil \le  \prod_{j=1}^{j_*}  \left \lceil \frac{\sqrt{C}}{j^{n/2}\delta}\right\rceil \le  \prod_{j=1}^{j_*}  \frac{2 \sqrt{C}}{j^{n/2}\delta} = \left(\frac{2 \sqrt{C}}{\delta}\right)^{j_*} \frac{1}{(j_*!)^{n/2}} \\
& = \left( \left( \left(\frac{2 \sqrt{C}}{\delta}\right)^{2/n}\right)^{j_*} \frac{1}{j_*!} \right)^{n/2} \le \left(\exp\left(\left(\frac{2 \sqrt{C}}{\delta}\right)^{2/n} \right)\right)^{n/2}.
\end{align*}
Hence
$$\log N_{\bsgamma}(\varepsilon,d) \le \frac{n}{2} \left(\frac{2 \sqrt{C}}{\delta}\right)^{2/n}=\frac{n}{2} \left(\frac{C}{S^2}\right)^{1/n} \varepsilon^{-2/n}.$$ Since $2/n<t_1$ we obtain $$\lim_{\varepsilon^{-1}+d \rightarrow \infty} \frac{\frac{n}{2} \left(\frac{C}{S^2}\right)^{1/n} \varepsilon^{-2/n}}{\varepsilon^{-t_1}+d^{t_2}}=0$$ and hence we achieve UWT. 
\end{enumerate}
\end{proof}

\begin{example}\rm
\begin{enumerate}
\item Let $\gamma_j=j^{-\alpha}$ for all $j \in\NN$. Then the centered regular grid with different mesh-sizes achieves WT for the weighted star discrepancy if and only if $\alpha>1$. It is not possible to achieve UWT, regardless of how big $\alpha$ is.
\item If $\gamma_j=\omega^{j^{\alpha}}$ with $\omega \in (0,1)$ and $\alpha>0$, then the centered regular grid with different mesh-sizes achieves UWT. One can even show that 
\begin{equation}\label{almostQPT}
N_{\bsgamma}(\varepsilon,d)\le \exp\left(\frac{1}{(\log 1/\sqrt{\omega})^{1/\alpha}} \left( \log \frac{S}{\varepsilon}\right)^{1+\frac{1}{\alpha}}\right),\ \ \ \mbox{where $S=\sum_{j \ge 1}\sqrt{\gamma_j}$.}
\end{equation}
\begin{proof}
For $j \in [d]$ choose $m_j=\lceil \sqrt{\gamma_j}/\delta \rceil,$ where $\delta:=\frac{2 \varepsilon}{S}$ and $S=\sum_{j \ge 1}\sqrt{\gamma_j} < \infty$ and QMC coefficients. Then we have $D_{N,\bsgamma}^*(\Gamma_{m_1,\ldots,m_d}) \le \varepsilon$. Furthermore, $m_j=1$ if $$\frac{\sqrt{\gamma_j}}{\delta} \le 1\ \Leftrightarrow \ j \ge \left(\log_x\frac{1}{\delta}\right)^{1/\alpha},$$ where $x=\omega^{-1/2}$ and $\log_x$ the logarithm to base $x$. Let $j_0(\varepsilon)=\lfloor \left( \log_x \delta^{-1}\right)^{1/\alpha}\rfloor.$ If we assume that $0 < \varepsilon < S \sqrt{\omega}/2$, then $\log_x \delta^{-1} \ge 1$ and hence $j_0(\varepsilon) \ge \frac{1}{2}( \log_x \delta^{-1})^{1/\alpha}$.

Now we have $$N =\prod_{j=1}^{j_0(\varepsilon)} m_j = \prod_{j=1}^{j_0(\varepsilon)} \left\lceil \frac{\sqrt{\gamma_j}}{\delta}\right\rceil \le \prod_{j=1}^{j_0(\varepsilon)} 2 \frac{\sqrt{\gamma_j}}{\delta},$$ since $\frac{\sqrt{\gamma_j}}{\delta} \ge 1$ for $j \in \{1,\ldots,j_0(\varepsilon)\}$. Hence $$N \le \left(\frac{2}{\delta}\right)^{j_0(\varepsilon)} (\sqrt{\omega})^{\sum_{j=1}^{j_0(\varepsilon)} j^{\alpha}} \le \left(\frac{2}{\delta}\right)^{j_0(\varepsilon)}$$ and therefore
\begin{equation*}
N_{\bsgamma}(\varepsilon,d)\le N \le  \left(\frac{S}{\varepsilon}\right)^{\frac{1}{(\log x)^{1/\alpha}} \left(\log \frac{S}{2 \varepsilon}\right)^{1/\alpha}}=\exp\left(\frac{1}{(\log x)^{1/\alpha}} \left( \log \frac{S}{\varepsilon}\right)^{1+\frac{1}{\alpha}}\right).
\end{equation*}
\end{proof}
\end{enumerate}
\end{example}

Note that the bound in \eqref{almostQPT} is very close to the definition of QPT in \eqref{def:QPT}. The only annoying term is the exponent $1+\frac{1}{\alpha}$. This leads to the question if also QPT is attainable for sufficiently fast decaying weights? However, the following result shows, that QPT, PT or SPT can only be achieved in the trivial case where the weight sequence $(\gamma_j)_{j \ge 1}$ becomes eventually zero.
\begin{theorem}\label{thm:PTiff}
The following assertions are equivalent:
\begin{enumerate}
\item the centered regular grid with different mesh-sizes achieves QPT for the $\bsgamma$-weighted star discrepancy;
\item the centered regular grid with different mesh-sizes achieves PT for the $\bsgamma$-weighted star discrepancy;
\item the centered regular grid with different mesh-sizes achieves SPT for the $\bsgamma$-weighted star discrepancy;
\item the weight sequence $(\gamma_j)_{j \ge 1}$ becomes eventually zero.
\end{enumerate}
\end{theorem}

\begin{proof}
As in the proof of Theorem~\ref{thm:WTiff} we have $$N_{\bsgamma}(\varepsilon,d) \ge \left(\frac{\gamma_d}{4 \varepsilon}\right)^d.$$ Assume that we have QPT, i.e., there are $C,t>0$ such that $$N_{\bsgamma}(\varepsilon,d) \le C \exp(t(1+\log d)(1+\log \varepsilon^{-1}))\ \ \ \mbox{ for all $d \in \NN$ and for all $\varepsilon \in (0,1)$.}$$ Then we have 
\begin{eqnarray*}
\gamma_d  \le 4 \varepsilon C^{1/d} {\rm e}^{\frac{t}{d}(1+\log d) (1+\log \varepsilon^{-1})}  = 4 C^{1/d} {\rm e}^{t \frac{1+\log d}{d}} {\rm e}^{\left(\tfrac{t}{d}(1+\log d) -1\right) \log \varepsilon^{-1}}.
\end{eqnarray*}
Since $\lim_{d \rightarrow \infty} (\tfrac{t}{d}(1+\log d) -1)=-1$ it follows that there is a $\tau>0$ such that $$\frac{t}{d}(1+\log d) -1 \le -\frac{1}{2}\ \ \ \mbox{ for all $d \ge \tau$.}$$ Hence for $d \ge \tau$ and for all $\varepsilon \in (0,1)$ we have $$\gamma_d \le 4 C^{1/d} {\rm e}^{t \frac{1+\log d}{d}} \sqrt{\varepsilon}.$$ Hence, $\gamma_d=0$ for all $d \ge \tau$ since $\varepsilon\in (0,1)$ can be chosen arbitrary small.

Now assume that $\gamma_j=0$ for all $j >\tau$. Choose $m_j=\lceil\tau/(2\varepsilon)\rceil$ for $j=1,2,\ldots,\tau$ and $m_j=1$ for $j > \tau$ and QMC coefficients. Then, for $d > \tau$ we have $$D_{N,\bsgamma}^*(\Gamma_{m_1,\ldots,m_d})\le \varepsilon\ \ \mbox{ and }\ \ N_{\bsgamma}(\varepsilon,d)\le N=m_1 m_2 \cdots m_d =\left\lceil\frac{\tau}{2\varepsilon}\right\rceil^{\tau}.$$ Thus, the centered regular grid with different mesh-sizes achieves SPT for the $\bsgamma$-weighted star discrepancy. Of course, SPT implies PT which in turn implies QPT.
\end{proof}

\paragraph{Literature review.}
There are many existence results and constructions of point sets which achieve tractability for the $\bsgamma$-weighted star discrepancy. First we mention that even the unweighted star discrepancy is polynomially tractable; this is a seminal result by Heinrich, Novak, Wasilkowski and Wo\'{z}niakowski \cite{HNWW} which opened a new stream of research in discrepancy theory. However, a concrete point set for which PT is achieved is still unknown. In the weighted case, there is a component-by-component construction of polynomial lattice points which yield SPT of the weighted star discrepancy whenever the weight sequence is summable; see \cite{DLP}. An existence result of point sets achieving SPT for $D_{N,\bsgamma}^{\ast}$ under the very mild condition $\sum_{j \ge 1} {\rm e}^{-c \gamma_j^{-2}} < \infty$ for some $c>0$ is shown in \cite{Ai}. 
In \cite{DGPW,DP} explicit point sets that achieve PT for  $D_{N,\bsgamma}^{\ast}$ are given. For a more extended literature review see \cite[Section~1]{DP} and for further background on tractability issues of (weighted) discrepancy see \cite{NW08,NW10}.

\paragraph{Acknowledgement.}
A former version of this manuscript only dealt with the QMC case. I thank Henryk Wo\'{z}niakowski for making me aware of a possible extension of the results to the setting with arbitrary coefficients.

\begin{small}
\noindent\textbf{Authors' address:} Friedrich Pillichshammer, Institut f\"{u}r Finanzmathematik und Angewandte Zahlentheorie, Johannes Kepler Universit\"{a}t Linz, Altenbergerstr.~69, 4040 Linz, Austria.\\
\textbf{E-mail:} \texttt{friedrich.pillichshammer@jku.at}
\end{small}

\end{document}